\documentclass[12pt,twoside]{amsart}
\usepackage{amsmath}
\usepackage{amsthm}
\usepackage{amsfonts}
\usepackage{amssymb}
\usepackage{latexsym}
\usepackage[all]{xy}
\usepackage{epsfig}
\usepackage{amscd,mathrsfs,graphicx,mathrsfs}
\usepackage[numbers,sort&compress]{natbib}

\date{}
\allowdisplaybreaks[4] \footskip=15pt

\renewcommand{\uppercasenonmath}[1]{}

\numberwithin{equation}{section} \theoremstyle{plain}
\newtheorem*{theorem*}{Theorem A}
\newtheorem*{theorem**}{Theorem B}
\newtheorem{theorem}{Theorem}[section]
\newtheorem{corollary}[theorem]{Corollary}
\newtheorem*{corollary*}{Corollary}

\newtheorem*{lemma*}{Lemma}
\newtheorem{proposition}[theorem]{Proposition}
\newtheorem*{proposition*}{Proposition}
\newtheorem{remark}[theorem]{Remark}
\newtheorem*{remark*}{Remark}

\newtheorem*{example*}{Example}
\newtheorem{definition}[theorem]{Definition}
\newtheorem*{definition*}{Definition}

\newtheorem*{conjecture*}{Conjecture}
\newtheorem*{ack*}{ACKNOWLEDGEMENTS}

\newcommand{\pf}{\noindent\begin {proof}}
\newcommand{\epf}{\end{proof}}

\begin{document}
\begin{center}
{\Large \bf Gorenstein projective and injective dimensions over Frobenius extensions
 \footnotetext{
Supported by the National Nature Science Foundation of China (11401476, 11561039) and a project of Chongqing Normal University (18XLB001).

E-mail address: wren@cqnu.edu.cn.}}

\vspace{0.5cm}    Wei Ren\\
{\small School of Mathematical Sciences, Chongqing Normal University, Chongqing 401331, China}
\end{center}


\bigskip
\centerline { \bf  Abstract}
\leftskip10truemm \rightskip10truemm
Let $R\subset A$ be a Frobenius extension of rings. We prove that: (1) for any left $A$-module $M$, $_{A}M$ is Gorenstein projective (injective) if and only if the underlying left $R$-module $_{R}M$ is Gorenstein projective (injective). (2) if $\mathrm{G}\text{-}\mathrm{proj.dim}_{A}M<\infty$, then $\mathrm{G}\text{-}\mathrm{proj.dim}_{A}M = \mathrm{G}\text{-}\mathrm{proj.dim}_{R}M$; the dual for Gorenstein injective dimension also holds.   (3) if the extension is split, then $\mathrm{G}\text{-}\mathrm{gldim}(A)= \mathrm{G}\text{-}\mathrm{gldim}(R)$.\\

\bigskip


\leftskip0truemm \rightskip0truemm \vbox to 0.2cm{}

\section { \bf Introduction}

The study of Gorenstein homological algebra stems from finitely generated modules of G-dimension zero over any noetherian rings, introduced by Auslander \cite{AB69} as a generalization of finite generated projective modules. In order to complete the analogy, Enochs and Jenda introduced the Gorenstein projective modules (not necessarily finitely generated) over any associative rings; and dually, Gorenstein injective modules are defined \cite{EJ95}. Then, Gorenstein homological dimensions are defined in the standard way, by using resolutions of Gorenstein modules. As shown in \cite[Theorem 4.2.6]{Chr00}, the Gorenstein projective dimension of a finitely generated module over a commutative noetherian ring agrees with its G-dimension. For finitely generated Gorenstein projective modules, there are several different terminologies in the literature, such as maximal Cohen-Macaulay modules, totally reflexive modules and modules of G-dimension zero.

In this paper, we intend to study the properties of Gorenstein projective (injective) modules and Gorenstein homological dimensions along Frobenius extension of rings. The theory of Frobenius extensions was developed by Kasch \cite{Kas54} as a generalization of Frobenius algebras, and was further studied by Nakayama-Tsuzuku \cite{NT60} and Morita \cite{Mor65} et. al. For example, the integral group ring extension $\mathbb{Z}\subset\mathbb{Z}G$ for a finite group $G$, and the ring extension of dual numbers of an algebra $R\subset R[x]/(x^2)$, are Frobenius extensions. There are other examples include Hopf subalgebras, finite extensions of enveloping algebras of Lie super-algebras, enveloping algebras of Lie coalgebras etc. We refer to a lecture due to Kadison \cite{Kad99} for more details.

We are motivated by a question raised by Chen \cite{Chen13}. He introduced a generalization of Frobenius extension, called the totally reflexive extension of rings, and proved that totally reflexive modules transfer along such extension. He asked if this is true for not necessarily finitely generated Gorenstein projective modules, and claimed that a new method for argument is needed. Our first main result gives an affirmative answer to this problem for Frobenius extensions. In Theorem \ref{thm 2.1}, we show that: for any Frobenius extension $R\subset A$ and any left $A$-module $M$, $_{A}M$ is Gorenstein projective in $\mathrm{Mod}(A)$ if and only if the underlying left $R$-module $_{R}M$ is Gorenstein projective; and a similar result for Gorenstein injective modules also holds.

In the study of Gorenstein homological algebra, an interesting assertion that ``every result in classical homological algebra has a counterpart in Gorenstein homological algebra'' (see Holm's thesis \cite{Holm}), can be confirmed by many works, see for example \cite{Hol04, EJ00, Chr00}. As a contribution to this, we give a Gorenstein counterpart of the result due to Nakayama-Tsuzuku (\cite[Theorem 8, Theorem 8$^{'}$]{NT60}): Let $R\subset A$ be a Frobenius extension, $M$ be any left $A$-module. If $\mathrm{G}\text{-}\mathrm{proj.dim}_{A}M<\infty$, then $\mathrm{G}\text{-}\mathrm{proj.dim}_{A}M = \mathrm{G}\text{-}\mathrm{proj.dim}_{R}M$; the dual for Gorenstein injective dimension also holds; see Proposition \ref{prop 3.1}.

For any ring $\Lambda$, Bennis and Mahdou proved an equality (\cite[Theorem 1.1]{BM10}):
\begin{center}$\mathrm{sup}\{\mathrm{G}\text{-}\mathrm{proj.dim}_{\Lambda}M \mid M\in \mathrm{Mod}(\Lambda) \} = \mathrm{sup}\{\mathrm{G}\text{-}\mathrm{inj.dim}_{\Lambda}M \mid M\in \mathrm{Mod}(\Lambda)\}.$\end{center}
As a Gorenstein counterpart of the global dimension, they named the common value of this equality the Gorenstein global dimension of $\Lambda$, and denoted it by $\textrm{G-gldim}(\Lambda)$. Following \cite{BR07}, a ring $\Lambda$ is left-Gorenstein provided that the category $\mathrm{Mod}(\Lambda)$ of left $\Lambda$-modules is a Gorenstein category. This is equivalent to the condition that the Gorenstein global dimension of $\Lambda$ is finite. According to a classical result established by Auslander, Buchsbaum and Serre, a commutative noetherian local ring is regular if and only if the projective dimension of its residue field is finite; moreover, in this case the ring has finite global dimension. So left-Gorenstein rings may be called (left) Gorenstein regular rings, meaning a Gorenstein counterpart of regular rings.

Let $R\subset A$ be any Frobenius extension. If the extension is split (i.e. $R$ is a direct summand of $A$ as an $R$-bimodule), then we prove in Theorem \ref{thm 3.1} that $A$ is Gorenstein regular if and only if $R$ is Gorenstein regular; and moreover, we show that $\mathrm{G}\text{-}\mathrm{gldim}(A)= \mathrm{G}\text{-}\mathrm{gldim}(R)$, that is, the Gorenstein global dimensions are invariant along Frobenius extensions; see Theorem \ref{thm 3.2}. Consequently, it follows immediately from  \cite[Theorem 4.1]{Emm12} that for a Gorenstein regular Frobenius extension $R\subset A$ (i.e. either $R$ or $A$ is Gorenstein regular), there are equalities: $\mathrm{spli}(A)= \mathrm{silp}(A)=\mathrm{fin.dim}(A) = \mathrm{spli}(R)  = \mathrm{silp}(R) = \mathrm{fin.dim}(R)$; see Corollary \ref{cor 3.1}. Here, the supremum of the projective lengths of injective left $R$-modules $\mathrm{spli}(R)$, and the supremum of the injective lengths of projective left $R$-modules $\mathrm{silp}(R)$, are two invariants introduced by Gedrich and Gruenberg \cite{GG87} in connection with the existence of complete cohomological functors in the category of left $R$-modules. The left finitistic dimension $\mathrm{fin.dim}(R)$ of $R$ is defined as the supremum of the projective dimensions of those left $R$-modules that have finite projective dimension.\\

The paper is organized as follows. In Section 2, we prove the first main result on transfer of Gorenstein projective and injective modules along Frobenius extensions, see Theorem \ref{thm 2.1} and \ref{thm 2.2}. The result in Theorem \ref{thm 2.1} gives an affirmative answer to Chen's question. In Section 3, we study Gorenstein homological dimensions along Frobenius extensions. A Gorenstein counterpart of \cite[Theorem 8 and 8$^{'}$]{NT60} is given in Proposition \ref{prop 3.1}, which shows that under the finiteness condition, the Gorenstein projective and injective dimensions of modules are invariant under Frobenius extensions. Then, we show that the Gorenstein regular property of rings (i.e.  finiteness of Gorenstein global dimension), and furthermore, the Gorenstein global dimension, are invariant along split Frobenius extensions; see Theorem \ref{thm 3.1} and \ref{thm 3.2}. Consequently, some equalities follows; see Corollary \ref{cor 3.1}.

\section{\bf Gorenstein projective and injective modules over Frobenius extensions}

Let $R$ be a ring. Recall that an $R$-module $M$ is said to be Gorenstein projective if $M$ is a syzygy of a totally acyclic complex of projective modules, i.e. if there exists an acyclic complex of projective $R$-modules $\mathbf{P}:=\cdots \rightarrow P_{1}\rightarrow P_{0}\rightarrow P_{-1}\rightarrow \cdots$ which remains acyclic when applying the functor $\mathrm{Hom}_{R}(-,P)$ for any projective $R$-module $P$, such that $M = \mathrm{Ker}(P_{0}\rightarrow P_{-1})$. Dually, Gorenstein injective modules are defined \cite{EJ00}. The study of Gorenstein homological algebra has found interesting applications in some areas such as representation theory, Tate cohomology and the theory of singularity categories, see for example \cite{AM02, Chen11, Hap91, Zhang11}. Moreover, it may prove to be useful in studying certain group-theoretical problems, such as characterizing algebraically the groups that admit a finite dimensional model for the classifying space for proper actions (\cite{BDT09}).

Throughout, all rings are associative with a unit. Homomorphisms of rings are required to send the unit to the unit. A left $R$-module $M$ is sometimes written as $_{R}M$. For two left $R$-modules $M$ and $N$, denote by $\mathrm{Hom}_{R}(M, N)$ the abelian
group consisting of left $R$-homomorphisms between them. A right $R$-module $M$ is sometimes written as $M_{R}$. We identify right $R$-modules
with left $R^{op}$-modules, where $R^{op}$ is the opposite ring of $R$. For two right $R$-modules $M$ and $N$, the abelian group of right $R$-homomorphisms is denoted by $\mathrm{Hom}_{R^{op}}(M, N)$. We denote by $\mathrm{Mod}(R)$ the category of left $R$-modules, and
$\mathrm{Mod}(R^{op})$ the category of right $R$-modules. Let $S$ be another ring. An $R$-$S$-bimodule $M$ is written as $_{R}M_{S}$.

We always denote a ring extension $\iota: R\hookrightarrow A$ by $R\subset A$. The natural bimodule $_{R}A_{R}$ is given by
$rar^{'}:= \iota(r)\cdot a\cdot \iota(r^{'})$. Similarly, we consider $_{R}A$ and $_{R}A_{A}$ etc.

The theory of Frobenius extensions was developed by Kasch \cite{Kas54} as a generalization of Frobenius algebras. Since then, Nakayama-Tsuzuku \cite{NT60} and Morita \cite{Mor65} et. al. defined natural generalizations of Frobenius extensions of different kinds. The definition of Frobenius extension we chose is a condition in \cite{Mor65}. We refer to \cite[Definition 1.1, Theorem 1.2]{Kad99} for the following definition of Frobenius extensions.

A functor between abelian categories is generally called ``Frobenius'' if it has left and right adjoints which are naturally equivalent.
For a ring extension $R\subset A$, there is a restricted functor $Res: \mathrm{Mod}(A)\rightarrow \mathrm{Mod}(R)$ sends $_{A}M$ to $_{R}M$. In the opposite direction, there are functors $T = A\otimes_{R}-: \mathrm{Mod}(R)\rightarrow \mathrm{Mod}(A)$ and $H= \mathrm{Hom}_{R}(A, -): \mathrm{Mod}(R)\rightarrow \mathrm{Mod}(A)$. It is clear that $(T, Res)$ and $(Res, H)$ are adjoint pairs.

\begin{definition}\label{def 1} A ring extension $R\subset A$ is a Frobenius extension, provided that one of the following equivalent conditions holds:\\
\indent $(1)$ The functors $T = A\otimes_{R}-$ and $H= \mathrm{Hom}_{R}(A, -)$ are naturally equivalent.\\
\indent $(2)$ $_{R}A$ is finite generated projective and $_{A}A_{R}\cong (_{R}A_{A})^{*}= \mathrm{Hom}_{R}(_{R}A_{A}, R)$.\\
\indent $(3)$ $A_{R}$ is finite generated projective and $_{R}A_{A}\cong (_{A}A_{R})^{*}= \mathrm{Hom}_{R^{op}}(_{A}A_{R}, R)$.\\
\indent $(4)$ There exists an $R$-$R$-homomorphism $\tau: A\rightarrow R$ and elements $x_i$, $y_i$ in $A$, such that for any $a\in A$, one has
$\sum\limits_{i}x_i\tau(y_ia) = a$ and $\sum\limits_{i}\tau(ax_i)y_i = a$.
\end{definition}

By \cite[Proposition 1]{NT60}, if we choose the automorphism of $R$ to be the identity of $R$, then the above definition coincides with the 2. Frobenius extension (or a Frobenius extension of 2nd kind) introduced by Nakayama-Tsuzuku.

There is an observation due to Buchweitz: for a finite group $G$, a $\mathbb{Z}G$-module, or equivalently an integral representation of $G$, is maximal Cohen-Macaulay over $\mathbb{Z}G$  if and only if the underlying $\mathbb{Z}$-module is maximal Cohen-Macaulay,  or equivalently, the underlying $\mathbb{Z}$-module is free, see \cite[Section 8.2]{Buc87}. Note that the classical example $\mathbb{Z}\subset \mathbb{Z}G$ is a Frobenius extension. In \cite{Chen13}, Chen introduced a generalization of Frobenius extension, called the totally reflexive extension of rings, and proved that totally reflexive modules transfer along such extension. He asked if this is true for not necessarily finitely generated Gorenstein projective modules, and claimed that a new method is needed for the question. We have the following, which gives an affirmative answer of Chen's problem in the case of Frobenius extensions. Moreover, it generalizes \cite[Theorem 2.5 and 2.11]{Ren}.

\begin{theorem}\label{thm 2.1} Let $R\subset A$ be a Frobenius extension of rings, $M$ be any left $A$-module. The following are equivalent:\\
\indent $(1)$ $_{A}M$ is Gorenstein projective in $\mathrm{Mod}(A)$.\\
\indent $(2)$ The underlying left $R$-module $_{R}M$ is Gorenstein projective.\\
\indent $(3)$ $A\otimes_{R}M$ and $\mathrm{Hom}_{R}(A, M)$ are Gorenstein projective left $A$-modules.
\end{theorem}

\begin{proof}
(1)$\Longrightarrow$(2). It follows from \cite[Lemma 2.2]{Ren}. Indeed, for the Gorenstein projective left $A$-module $M$,  there exists a totally acyclic complex of projective $A$-modules $\mathbf{P}:=\cdots \rightarrow P_{1}\rightarrow P_{0}\rightarrow P_{-1}\rightarrow \cdots$ such that
$M = \mathrm{Ker}(P_{0}\rightarrow P_{-1})$. By restricting $\mathbf{P}$ one gets an acyclic complex of projective $R$-modules. For any projective left $R$-module $Q$,  $\mathrm{Hom}_{R}(A, Q)\cong A\otimes_{R}Q$ is a projective $A$-module, and it follows from isomorphisms $\mathrm{Hom}_{R}(\mathbf{P}, Q)\cong\mathrm{Hom}_{R}(A\otimes_{A}\mathbf{P}, Q)\cong \mathrm{Hom}_{A}(\mathbf{P}, \mathrm{Hom}_{R}(A, Q))$ that $\mathrm{Hom}_{R}(\mathbf{P}, Q)$ is acyclic. The assertion follows.

(2)$\Longrightarrow$(3). Let $\mathbf{P}:=\cdots \rightarrow P_{1}\rightarrow P_{0}\rightarrow P_{-1}\rightarrow \cdots$ be a totally acyclic complex of projective $R$-modules such that $_{R}M = \mathrm{Ker}(P_{0}\rightarrow P_{-1})$. It is direct to check that $A\otimes_{R}\mathbf{P}$ is a totally acyclic complex of projective $A$-modules, and $A\otimes_{R}M = \mathrm{Ker}(A\otimes_{R}P_{0}\rightarrow A\otimes_{R}P_{-1})$. Hence, $A\otimes_{R}M$ and $\mathrm{Hom}_{R}(A, M)$ are Gorenstein projective left $A$-modules.

(3)$\Longrightarrow$(2). Note that for the ring extension $R\subset A$ and any $A$-module $M$, the module $_{R}M$ is a direct summand of the $R$-module $A\otimes_{R}M$. If $A\otimes_{R}M$ is a Gorenstein projective left $A$-module, then since (1) implies (2), we get that $A\otimes_{R}M$ is Gorenstein projective in $\mathrm{Mod}(R)$, and hence $_{R}M$ is Gorenstein projective.

(3)$\Longrightarrow$(1). Let $P$ be any projective $A$-module. Since (3) implies (2), we have, from $A\otimes_{R}M$ being a Gorenstein projective left $A$-module, that the module $_{R}M$ is Gorenstein projective. Then it follows from \cite[Lemma 2.3]{Ren} that $\mathrm{Ext}_{A}^{i}(M,P)=0$. Indeed, note that the module $_{R}P$ is projective, and then we have $0= \mathrm{Ext}_{R}^{i}(M,P)\cong \mathrm{Ext}_{R}^{i}(A\otimes_{A}M, P)\cong \mathrm{Ext}_{A}^{i}(M,\mathrm{Hom}_{R}(A,P))\cong\mathrm{Ext}_{A}^{i}(M, A\otimes_{A}P)$. Moreover, since $_{A}P$ is a direct summand of $A\otimes_{R}P$, and then $\mathrm{Ext}_{A}^{i}(M,P)=0$. It only remains to construct the right part of the totally acyclic complex of $_{A}M$.

Since $\mathrm{Hom}_{R}(A,M)$ is a Gorenstein projective $A$-module, there is an exact sequence
$0\rightarrow \mathrm{Hom}_{R}(A,M)\stackrel{f}\rightarrow P_{0}\rightarrow L\rightarrow 0$ of $A$-modules, where $P_{0}$ is projective and $L$ is Gorenstein projective. There is a map $\varphi: M\rightarrow\mathrm{Hom}_{R}(A,M)$ given by $\varphi(m)(a)=am$, which is an $A$-monomorphism, and is split when we restrict it as an $R$-homomorphism. So we have an $R$-homomorphism $\varphi^{'}:  \mathrm{Hom}_{R}(A,M)\rightarrow M$ such that $\varphi^{'}\varphi=\mathrm{id}_{M}$. Let $Q$ be any projective $R$-module, and $g: M\rightarrow Q$ be any $R$-homomorphism. Since $L$ is also Gorenstein projective as an $R$-module, for the $R$-homomorphism $g\varphi^{'}: \mathrm{Hom}_{R}(A,M)\rightarrow Q$, there is an $R$-homomorphism $h: P_{0}\rightarrow Q$, such that $g\varphi^{'}= hf$. That is, we have the following commutative diagram:
$$\begin{xymatrix}{
 &Q\\
0 \ar[r]^{}  &\mathrm{Hom}_{R}(A,M) \ar[u]^{g\varphi^{'}} \ar[r]^{\quad\quad f} &P_{0}\ar[lu]_{\exists h}\ar[r] &L\ar[r] &0
}\end{xymatrix}$$

Now we have an $A$-monomorphism $f\varphi: M\rightarrow P_{0}$. Consider the exact sequence $0\rightarrow M\stackrel{f\varphi}\rightarrow P_{0}\rightarrow L_{0}\rightarrow 0$ of $A$-modules, where $P_{0}$ is projective, and $L_{0} =\mathrm{Coker}(f\varphi)$. Restricting the sequence, we note that it is $\mathrm{Hom}_{R}(-,Q)$-exact for any projective $R$-module $Q$, since for any $R$-homomorphism $g: M\rightarrow Q$, there exists an $R$-homomorphism $h: P_{0}\rightarrow Q$ such that $g=(g\varphi^{'})\varphi=h(f\varphi)$. Then, it follows from the exact sequence $\mathrm{Hom}_{R}(P_{0},Q)\rightarrow\mathrm{Hom}_{R}(M,Q)\rightarrow\mathrm{Ext}_{R}^{1}(L_{0},Q)\rightarrow 0$ that $\mathrm{Ext}_{R}^{1}(L_{0},Q)=0$. Moreover, $_{R}M$ is Gorenstein projective by (3)$\Rightarrow$(2) and $_{R}P_{0}$ is projective, it follows from \cite[Corollary 2.11]{Hol04} that $L_{0}$ is a Gorenstein projective $R$-module.

Let $P$ be any projective $A$-module. There is a split epimorphism $\psi: A\otimes_{R}P\rightarrow P$ of $A$-modules given by $\psi(a\otimes_{R}x)=ax$ for any $a\in A$ and $x\in P$, and then there exists an $A$-homomorphism $\psi^{'}: P\rightarrow A\otimes_{R}P$ such that $\psi\psi^{'}=\mathrm{id}_{P}$. Note that $P$ is also projective as an $R$-module. Then, it follows from $\mathrm{Ext}_{A}^{1}(L_{0}, A\otimes_{R}P)\cong \mathrm{Ext}_{R}^{1}(L_{0}, P) = 0$ that the exact sequence $0\rightarrow M\stackrel{f\varphi}\rightarrow P_{0}\rightarrow L_{0}\rightarrow 0$ remains exact after applying $\mathrm{Hom}_{A}(-, A\otimes_{R}P)$.

For any $A$-homomorphism $\alpha: M\rightarrow P$, we consider the following diagram
$$\begin{xymatrix}{
 &P\ar[r]^{\psi^{'}\quad} &A\otimes_{R}P\\
 0 \ar[r]^{}  &M\ar[u]^{\alpha}\ar[r]^{f\varphi} &P_{0}\ar[r]\ar[u]_{\exists \beta}\ar@{-->}[lu]_{} &L_{0}\ar[r] &0
}\end{xymatrix}$$
For $\psi^{'}\alpha: M\rightarrow A\otimes_{R}P$, there exists an $A$-map $\beta: P_{0}\rightarrow A\otimes_{R}P$ such that $\psi^{'}\alpha= \beta(f\varphi)$. And then, we have $\psi\beta: P_{0}\rightarrow P$, such that $\alpha = (\psi\psi^{'})\alpha=(\psi\beta)(f\varphi)$. This implies that the sequence  $0\rightarrow M\stackrel{f\varphi}\rightarrow P_{0}\rightarrow L_{0}\rightarrow 0$ is $\mathrm{Hom}_{A}(-, P)$-exact.

Note that $L_{0}$ is a Gorenstein projective $R$-module, and then $\mathrm{Hom}_{R}(A, L_{0})$ is a Gorenstein projective $A$-module. Repeating the process we followed with $M$, we inductively construct an exact sequence $0\rightarrow M\rightarrow P_{0}\rightarrow P_{1}\rightarrow P_{2}\rightarrow \cdots$ in $\mathrm{Mod}(A)$, with each $P_{i}$ projective and which is also exact after applying $\mathrm{Hom}_{A}(-,P)$ for any projective $A$-module $P$. This completes the proof.
\end{proof}

Dually, we have the following.

\begin{theorem}\label{thm 2.2} Let $R\subset A$ be a Frobenius extension of rings, $M$ be any left $A$-module. The following are equivalent:\\
\indent $(1)$ $_{A}M$ is Gorenstein injective in $\mathrm{Mod}(A)$.\\
\indent $(2)$ The underlying left $R$-module $_{R}M$ is Gorenstein injective.\\
\indent $(3)$ $A\otimes_{R}M$ and $\mathrm{Hom}_{R}(A, M)$ are Gorenstein injective left $A$-modules.
\end{theorem}

\section{\bf Gorenstein projective and injective dimensions over Frobenius extensions}

Unless otherwise mentioned we will be working with left modules. The Gorenstein projective and injective dimensions of modules are defined in the standard way by using resolutions of Gorenstein modules. That is, the Gorenstein projective dimension of a $\Lambda$-module $M$, denoted by $\text{G-proj.dim}_{\Lambda}M$, is defined by declaring that $\text{G-proj.dim}_{\Lambda}M \leq n$ ($n\in \mathbb{N}$) if, and only if, $M$ has a Gorenstein projective resolution $0\rightarrow G_{n}\rightarrow\cdots\rightarrow G_{1}\rightarrow G_{0}\rightarrow M\rightarrow 0$ of length $n$. We set $\text{G-proj.dim}_{\Lambda}M = \infty$ if there is no such a resolution. Similarly, the Gorenstein injective dimension is defined; see for example \cite{EJ00, Hol04}.

It follows from \cite[Theorem 8, Theorem 8$^{'}$]{NT60} that: for a Frobenius extension $R\subset A$ and any left $A$-module $M$, if the $A$-projective dimension ($A$-injective dimension, respectively) of $M$ is finite, then one has $\text{proj.dim}_{A} M = \text{proj.dim}_{R}M$ ($\text{inj.dim}_{A} M = \text{inj.dim}_{R}M$, respectively). We can extend this result to corresponding Gorenstein homological dimensions.
This serves as an example to support the metatheorem (see Holm's thesis \cite{Holm}) ``every result in classical homological algebra has a counterpart in Gorenstein homological algebra''.

\begin{proposition}\label{prop 3.1} Let $R\subset A$ be a Frobenius extension of rings. For any left $A$-module $M$, if $\mathrm{G}\text{-}\mathrm{proj.dim}_{A}M<\infty$, then
$\mathrm{G}\text{-}\mathrm{proj.dim}_{A}M = \mathrm{G}\text{-}\mathrm{proj.dim}_{R}M$. Dually, if $\mathrm{G}\text{-}\mathrm{inj.dim}_{A}M<\infty$, then $\mathrm{G}\text{-}\mathrm{inj.dim}_{A}M = \mathrm{G}\text{-}\mathrm{inj.dim}_{R}M$.
\end{proposition}

\begin{proof}
By Theorem \ref{thm 2.1}, any Gorenstein projective $A$-module is also Gorenstein projective in $R$-Mod. It is easy to see that $\mathrm{G}\text{-}\mathrm{proj.dim}_{R}M \leq \mathrm{G}\text{-}\mathrm{proj.dim}_{A}M$. For the converse, we can assume that $\mathrm{G}\text{-}\mathrm{proj.dim}_{R}M = n < \infty$. Let $P$ be any projective $A$-module. Then, $P$ is also projective as an $R$-module. By \cite[Theorem 2.20]{Hol04}, for any $i>0$ we have $\mathrm{Ext}_{R}^{n+i}(M, P)=0$. Moreover, since $\mathrm{Ext}_{A}^{n+i}(M, A\otimes_{R}P)\cong \mathrm{Ext}_{R}^{n+i}(M, P)$ and $P$ is a direct summand of $A\otimes_{R}P$ as $A$-modules, we have  $\mathrm{Ext}_{A}^{n+i}(M, P)=0$. This implies that $\mathrm{G}\text{-}\mathrm{proj.dim}_{A}M \leq n$. Then, the equality
$\mathrm{G}\text{-}\mathrm{proj.dim}_{A}M = \mathrm{G}\text{-}\mathrm{proj.dim}_{R}M$ holds. Analogously, we can prove the assertion for Gorenstein injective dimension.
\end{proof}

\begin{proposition}\label{prop 3.2} Let $R\subset A$ be a Frobenius extension of rings, and $M$ be any left $A$-module. Then $\mathrm{G}\text{-}\mathrm{proj.dim}_{R}M = \mathrm{G}\text{-}\mathrm{proj.dim}_{A}(A\otimes_{R}M) = \mathrm{G}\text{-}\mathrm{proj.dim}_{R}(A\otimes_{R}M)$. Similarly, we have $\mathrm{G}\text{-}\mathrm{inj.dim}_{R}M = \mathrm{G}\text{-}\mathrm{inj.dim}_{A}(A\otimes_{R}M) = \mathrm{G}\text{-}\mathrm{inj.dim}_{R}(A\otimes_{R}M)$.
\end{proposition}

\begin{proof}
It follows from Theorem \ref{thm 2.1} that $\mathrm{G}\text{-}\mathrm{proj.dim}_{R}(A\otimes_{R}M) \leq \mathrm{G}\text{-}\mathrm{proj.dim}_{A}(A\otimes_{R}M)$. For any Gorenstein projective $R$-module $G$, it follows from Theorem \ref{thm 2.1} that $A\otimes_{R}G$ is a Gorenstein projective $A$-module. Then $\mathrm{G}\text{-}\mathrm{proj.dim}_{A}(A\otimes_{R}M) \leq \mathrm{G}\text{-}\mathrm{proj.dim}_{R}M$ is easy to see. As $R$-modules, $M$ is a direct summand of $A\otimes_{R} M$. It follows immediately from \cite[Proposition 2.19]{Hol04} that $\mathrm{G}\text{-}\mathrm{proj.dim}_{R}M\leq \mathrm{G}\text{-}\mathrm{proj.dim}_{R}(A\otimes_{R}M)$. Hence, we get the desired equality.
\end{proof}

By \cite[Definition VII2.1, VII2.5]{BR07}, a ring $\Lambda$ is left-Gorenstein provided the category $\mathrm{Mod}(\Lambda)$ of left $\Lambda$-modules is a Gorenstein category, that is, if both $\mathrm{spli}(\Lambda)$ and $\mathrm{silp}(\Lambda)$ are finite. Here, $\mathrm{spli}(\Lambda)$ is the supremum of the projective lengths of injective left $\Lambda$-modules, and $\mathrm{silp}(\Lambda)$ is the supremum of the injective lengths of projective left $\Lambda$-modules. These two invariants are introduced by Gedrich and Gruenberg \cite{GG87}, in connection with the existence of complete cohomological functors in the category of left $\Lambda$-modules.

According to a classical result established by Auslander, Buchsbaum and Serre, a commutative noetherian local ring is regular if and only if the projective dimension of its residue field is finite; moreover, in this case the ring has finite global dimension. It is known that $\Lambda$ is left-Gorenstein if and only if the left Gorenstein global dimension of $\Lambda$ is finite, see for example \cite{Emm12}. We prefer to call left-Gorenstein ring as (left) Gorenstein regular ring, meant a Gorenstein counterpart of regular ring. By \cite[Theorem 10.2.14]{EJ00}, each Iwanaga-Gorenstein ring (i.e. two-sided noetherian ring with finite left and right self-injective dimension) is Gorenstein regular.

Recall that an extension $R\subset A$ of rings is split if $A= R\oplus S$ as $R$-bimodules for some subbimodule $S$ in $A$; see for example \cite{Kad}. In this case, it is clear to see that for any $R$-module $M$, $M= R\otimes_{R}M$ is a direct summand of $A\otimes_{R}M$.

\begin{theorem}\label{thm 3.1}
Let $R\subset A$ be a split Frobenius extension of rings. Then $A$ is Gorenstein regular if and only if $R$ is Gorenstein regular.
\end{theorem}

\begin{proof}
Let $R\subset A$ be a Frobenius extension. We claim that every projective $R$-module has finite $R$-injective dimension if, and only if every projective $A$-module has finite $A$-injective dimension.

For the ``only if'' part, let $P$ be a projective left $A$-module. Consider $P$ as an $R$-module, then by the assumption $\text{inj.dim}_{R}P$ is finite. Assume that $\text{inj.dim}_{R}P =n$ and $0\rightarrow P\rightarrow I^0\rightarrow I^1\rightarrow \cdots\rightarrow I^n\rightarrow 0$ is an $R$-injective resolution of $_{R}P$. For any injective $R$-module $I$, it follows from the isomorphism $A\otimes_{R}I\cong \mathrm{Hom}_{R}(A, I)$ that $A\otimes_{R}I$ is an injective left $A$-module. Hence, the exact sequence $0\rightarrow A\otimes_{R}P\rightarrow A\otimes_{R}I^0\rightarrow A\otimes_{R}I^1\rightarrow \cdots\rightarrow A\otimes_{R}I^n\rightarrow 0$ is an $A$-injective resolution of
$A\otimes_{R}P$. Moreover, $_{A}P$ is a direct summand of $A\otimes_{R}P$, so $P$ is of finite $A$-injective dimension.

Conversely, for the ``if'' part, let $Q$ be a projective left $R$-module. By the assumption, the projective left $A$-module $A\otimes_{R}Q$ has finite $A$-injective dimension. It follows from \cite[Theorem 8$^{'}$]{NT60} that $\text{inj.dim}_{R}(A\otimes_{R}Q) = \text{inj.dim}_{A}(A\otimes_{R}Q)<\infty$. Moreover, $Q$ is a direct summand of $A\otimes_{R}Q$, and then $Q$ has finite $R$-injective dimension.

Similarly, we can prove that every injective $R$-module has finite $R$-projective dimension if and only if every injective $A$-module has finite $A$-projective dimension. This will imply the desired assertion that $A$ is Gorenstein regular if and only if $R$ is Gorenstein regular.
\end{proof}

Moreover, we have the following. It shows that not only the finiteness of Gorenstein global dimension, but also  Gorenstein global dimension itself, is invariant under Frobenius extensions.

\begin{theorem}\label{thm 3.2}
Let $R\subset A$ be a split Frobenius extension of rings. Then $\mathrm{G}\text{-}\mathrm{gldim}(A)= \mathrm{G}\text{-}\mathrm{gldim}(R)$.
\end{theorem}

\begin{proof}
We deduce from Theorem \ref{thm 3.1} that $\mathrm{G}\text{-}\mathrm{gldim}(A)=\infty$ if and only if $\mathrm{G}\text{-}\mathrm{gldim}(R)=\infty$. Now we assume that both $\mathrm{G}\text{-}\mathrm{gldim}(A)$ and $\mathrm{G}\text{-}\mathrm{gldim}(R)$ are finite.

By Proposition \ref{prop 3.1}, there is an equality $\mathrm{G}\text{-}\mathrm{proj.dim}_{A}M = \mathrm{G}\text{-}\mathrm{proj.dim}_{R}M$ for any $A$-module $M$. Hence, $\mathrm{G}\text{-}\mathrm{gldim}(A)\leq \mathrm{G}\text{-}\mathrm{gldim}(R)$. Let $N$ be any $R$-module. By Proposition \ref{prop 3.1}, $\mathrm{G}\text{-}\mathrm{proj.dim}_{R}N \leq \mathrm{G}\text{-}\mathrm{proj.dim}_{R}(A\otimes_{R}N) = \mathrm{G}\text{-}\mathrm{proj.dim}_{A}(A\otimes_{R}N)$. This implies that $\mathrm{G}\text{-}\mathrm{gldim}(R) \leq \mathrm{G}\text{-}\mathrm{gldim}(A)$. Then, the desired equality follows.
\end{proof}

For a ring $\Lambda$ of finite Gorenstein global dimension (i.e. Gorenstein regular ring, or left-Gorenstein ring), Emmanouil got the following equalities:
$\mathrm{G}\text{-}\mathrm{gldim}(\Lambda)= \mathrm{spli}(\Lambda)=\mathrm{silp}(\Lambda)=\mathrm{fin.dim}(\Lambda)$, by comparing Gorenstein projective and injective dimensions with some invariants of rings; see \cite[Theorem 4.1]{Emm12}. Here, the left finitistic dimension $\mathrm{fin.dim}(\Lambda)$ of $\Lambda$ is defined as the supremum of the projective dimensions of those left $\Lambda$-modules that have finite projective dimension. The following is immediate.

\begin{corollary}\label{cor 3.1} Let $R\subset A$ be a split Frobenius extension of rings. If either $A$ or $R$ is Gorenstein regular, then there are equalities: $$\mathrm{spli}(A)= \mathrm{silp}(A)=\mathrm{fin.dim}(A) = \mathrm{spli}(R)  = \mathrm{silp}(R) = \mathrm{fin.dim}(R).$$
\end{corollary}

\begin{remark} \indent$(1)$ For any finite group $G$, the integer group ring extension $\mathbb{Z}\subset \mathbb{Z}G$ is a split Frobenius extension. For any ring $R$, $R\subset R[x]/(x^2)$ is a split Frobenius extension, where $x$ is a variable which is supposed to commutate with all the elements of $R$.\\
\indent$(2)$ Every excellent extension (see e.g. \cite{Ren19} for the collection of definition and examples) is a split Frobenius extension.\\
\indent$(3)$ Recall that $R\subset A$ is a Frobenius extension of rings if and only if there exists an $R$-$R$-homomorphism $\tau: A\rightarrow R$ and elements $x_i$, $y_i$ in $A$, such that for any $a\in A$, one has $\sum\limits_{i}x_i\tau(y_ia) = a$ and $\sum\limits_{i}\tau(ax_i)y_i = a$; the triple $(\tau, x_{i}, y_{i})$ is called a Frobenius system, $\tau$ a Frobenius homomorphism. It follows from \cite[Corollary 4.2]{Kad} that if $\tau(1)=1$ then the Frobenius extension $R\subset A$ is split. Indeed, for the $R$-homomorphism $\varphi: A\otimes_{R}R\rightarrow R$ given by $\varphi(a\otimes r)= \tau(a)r$ for any $a\in A$ and $r\in R$, we have $\psi: R\rightarrow A\otimes_{R}R$, $r\rightarrow re$ with $e= \sum\limits_{i}x_{i}\otimes \tau(y_i)$, such that the composition $\varphi\psi$ is the identity map of $R$. Moreover, the relation between split Frobenius extensions
and separable Frobenius extensions is studied by Kadison \cite[Proposition 4.1]{Kad}.
\end{remark}

\begin{ack*}
The author is grateful to the referee for several comments that improved the paper, and he thanks Professor Xiao-Wu Chen for helpful suggestions.
\end{ack*}

\end{document}